\documentclass[a4paper,11pt]{article}
\pdfoutput=1
\usepackage{hyperref}
\hypersetup{hypertexnames = false, bookmarksdepth = 2, bookmarksopen = true, colorlinks, linkcolor = black, citecolor = black, urlcolor = black, pdfstartview={XYZ null null 1}}

\usepackage{amsfonts}
\usepackage[fleqn, leqno]{amsmath}
\usepackage{amsthm}

\usepackage{booktabs}
\usepackage{subcaption}
\usepackage{colonequals}
\usepackage{enumitem}
\usepackage{fullpage}
\usepackage{mathtools}
\usepackage{parskip}
\usepackage{thmtools}
\usepackage{tikz-cd}
\usepackage{xparse}
\usepackage{xpatch}
\usepackage{xspace}

\usepackage[utf8]{inputenc}
\usepackage[T1]{fontenc}
\usepackage{libertine}
\usepackage[libertine]{newtxmath}
\usepackage[scaled=0.83]{beramono}
\usepackage{eucal}
\usepackage{microtype}
\usepackage{gitinfo2}

\usetikzlibrary{arrows.meta}
\usetikzlibrary{calc}

\usepackage[backend=biber, maxbibnames=99, datamodel=mrnumber, sortcites]{biblatex}

\usepackage[capitalise,noabbrev]{cleveref}

\DeclareFieldFormat{mrnumber}{%
  MR\addcolon\space
  \ifhyperref
    {\href{http://www.ams.org/mathscinet-getitem?mr=#1}{\nolinkurl{#1}}}
    {\nolinkurl{#1}}}

\renewbibmacro*{doi+eprint+url}{%
  \iftoggle{bbx:doi}
    {\printfield{doi}}
    {}%
  \newunit\newblock
  \printfield{mrnumber}%
  \newunit\newblock
  \iftoggle{bbx:eprint}
    {\usebibmacro{eprint}}
    {}%
  \newunit\newblock
  \iftoggle{bbx:url}
    {\usebibmacro{url+urldate}}
    {}}

\DeclareSourcemap{
  \maps[datatype=bibtex]{
    \map{
      \step[fieldset=issn, null]
    }
  }
}
\xpretobibmacro{bib+doi+url}
  {\iffieldundef{doi}
     {}
     {\clearfield{url}}}
  {}{}

\xpretobibmacro{cite+doi+url}
  {\iffieldundef{doi}
     {}
     {\clearfield{url}}}
  {}{}

\AtEveryCite{\upshape}

\declaretheoremstyle[spaceabove = 3pt, spacebelow = 0pt, bodyfont = \itshape]{theorem}
\declaretheoremstyle[spaceabove = 3pt, spacebelow = 0pt]{remark}

\declaretheorem[style=theorem]{theorem}
\declaretheorem[style=theorem, sibling=theorem]{corollary}
\declaretheorem[style=theorem, sibling=theorem]{lemma}
\declaretheorem[style=theorem, sibling=theorem]{proposition}

\declaretheorem[style=remark, sibling=theorem]{definition}

\declaretheorem[style=remark, sibling=theorem]{remark}

\declaretheorem[style=theorem, numberwithin=section, title=Theorem]{alphatheorem}

\crefname{alphatheorem}{Theorem}{Theorems}
\Crefname{alphatheorem}{Theorem}{Theorems}
\crefname{alphaconjecture}{Conjecture}{Conjectures}
\Crefname{alphaconjecture}{Conjecture}{Conjectures}
\crefname{alphacorollary}{Corollary}{Corollaries}
\Crefname{alphacorollary}{Corollary}{Corollaries}
\crefname{alphaproposition}{Proposition}{Propositions}
\Crefname{alphaproposition}{Proposition}{Propositions}

\setlist[enumerate]{font=\normalfont}

\mathchardef\mhyphen="2D

\newcommand\vect[1]{\ensuremath{\boldsymbol #1}}

\newcommand\QM[1]{%
  \ifcase#1
    \mathbf{?}%
    \or \texorpdfstring{\ensuremath{\mathrm{M}(1,2^4;4)}}{M(1,2,2,2,2;4)}%
    \or \texorpdfstring{\ensuremath{\mathrm{M}(1^5,2;3)}}{M(1,1,1,1,1,2;3)}%
    \or \texorpdfstring{\ensuremath{\mathrm{M}(1^4,2^2;3)}}{M(1,1,1,1,2,2;3)}%
    \or \texorpdfstring{\ensuremath{\mathrm{M}(1^7;2)}}{M(1,1,1,1,1,1,1;2)}%
    \else
    \mathbf{?}%
  \fi
}

\newcommand{\framed}{\mhyphen\mathrm{fr}}
\newcommand{\stable}{\mhyphen\mathrm{st}}

\newcommand{\semistable}{\mhyphen\mathrm{sst}}

\DeclareDocumentCommand\modulispace{om}{\IfNoValueTF{#1}{\mathrm{M}{(#2)}}{\mathrm{M}^{#1}(#2)}}
\DeclareDocumentCommand\repspace{om}{\IfNoValueTF{#1}{\mathrm{R}{(#2)}}{\mathrm{R}^{#1}(#2)}}

\newcommand\Aone{\ensuremath{\widetilde{\mathrm{A}}_1}}
\newcommand\AoneA{\ensuremath{\widetilde{\mathrm{A}}_1}(a)}
\newcommand\AoneB{\ensuremath{\widetilde{\mathrm{A}}_1}(b)}
\newcommand\Atwo{\ensuremath{\widetilde{\mathrm{A}}_2}}
\newcommand\Athree{\ensuremath{\widetilde{\mathrm{A}}_3}}
\newcommand\Dfour{\ensuremath{\widetilde{\mathrm{D}}_4}}

\newcommand\tildeA{\ensuremath{\widetilde{\mathrm{A}}}}
\newcommand\tildeD{\ensuremath{\widetilde{\mathrm{D}}}}
\newcommand\tildeE{\ensuremath{\widetilde{\mathrm{E}}}}

\newcommand\can{\ensuremath{\mathrm{can}}}
\newcommand\nilpotent{\ensuremath{\mathrm{nilp}}}

\DeclareMathOperator\Bl{Bl}
\DeclareMathOperator\dimvect{\mathbf{dim}}

\DeclareMathOperator\Exp{Exp}

\DeclareMathOperator\HH{H}
\DeclareMathOperator\Hom{Hom}
\DeclareMathOperator\Ker{Ker}

\DeclareMathOperator\Rep{Rep}
\DeclareMathOperator\Stab{Stab}
\DeclareMathOperator\supp{supp}
\DeclareMathOperator\Sym{Sym}

\tikzset{
  vertex/.style={draw, circle, minimum size=4pt, inner sep=0pt, outer sep=2pt},
  quiver/.style={baseline={(current bounding box.center)}},
}

\newcommand{\QAone}{\begin{tikzpicture}[quiver]
  \node[vertex,label=left:$i$]  (i) at (0,0) {};
  \node[vertex,label=right:$j$] (j) at (1,0) {};
  \draw[->] (i) to[bend left=16] (j);
  \draw[->] (i) to[bend right=16] (j);
\end{tikzpicture}}
\newcommand{\QAtwo}{\begin{tikzpicture}[quiver]
  \node[vertex,label=above:$i$] (i) at (0.5,0.87) {};
  \node[vertex,label=left:$j$]  (j) at (0,0) {};
  \node[vertex,label=right:$k$] (k) at (1,0) {};
  \draw[->] (i)--(j); \draw[->] (j)--(k); \draw[->] (i)--(k);
\end{tikzpicture}}
\newcommand{\QAthree}{\begin{tikzpicture}[quiver,scale=0.6]
  \node[vertex,label=above:$j$] (j) at (0,1) {};
  \node[vertex,label=left:$i$]  (i) at (-1,0) {};
  \node[vertex,label=right:$l$] (l) at (1,0) {};
  \node[vertex,label=below:$k$] (k) at (0,-1) {};
  \draw[->] (i)--(j); \draw[->] (l)--(j); \draw[->] (i)--(k); \draw[->] (l)--(k);
\end{tikzpicture}}
\newcommand{\QDfour}{\begin{tikzpicture}[quiver,scale=0.6]
  \node[vertex,label=right:$j$]         (c) at (0,0) {};
  \node[vertex,label=above left:$i_1$]  (a) at (-1,1) {};
  \node[vertex,label=below left:$i_2$]  (b) at (-1,-1) {};
  \node[vertex,label=above right:$i_3$] (d) at (1,1) {};
  \node[vertex,label=below right:$i_4$] (e) at (1,-1) {};
  \draw[->] (a)--(c); \draw[->] (b)--(c); \draw[->] (d)--(c); \draw[->] (e)--(c);
\end{tikzpicture}}

\newcommand{\QfrAoneA}{\begin{tikzpicture}[quiver]%
  \node[vertex] (i) at (0,0) {};
  \node[vertex] (j) at (1.2,0) {};
  \node[vertex] (f) at (2.4,0) {};
  \draw[->] (i) to[bend left=16] (j);
  \draw[->] (i) to[bend right=16] (j);
  \draw[->] (f) to[bend left=16] (j);
  \draw[->] (f) to[bend right=16] (j);
\end{tikzpicture}}
\newcommand{\QfrAoneB}{\begin{tikzpicture}[quiver]%
  \node[vertex] (i) at (0,0) {};
  \node[vertex] (j) at (1.4,0) {};
  \node[vertex] (f) at (0.7,1) {};
  \draw[->] (i) to[bend left=16] (j);
  \draw[->] (i) to[bend right=16] (j);
  \draw[->] (f)--(i); \draw[->] (f)--(j);
\end{tikzpicture}}
\newcommand{\QfrAtwo}{\begin{tikzpicture}[quiver]%
  \node[vertex] (i) at (0.5,0.87) {};
  \node[vertex] (j) at (0,0) {};
  \node[vertex] (k) at (1,0) {};
  \node[vertex] (f) at (0.5,-0.9) {};
  \draw[->] (i)--(j); \draw[->] (j)--(k); \draw[->] (i)--(k);
  \draw[->] (f)--(j); \draw[->] (f)--(k);
\end{tikzpicture}}
\newcommand{\QfrAthree}{\begin{tikzpicture}[quiver,scale=0.6]%
  \node[vertex] (j) at (0,1) {};
  \node[vertex] (i) at (-1,0) {};
  \node[vertex] (l) at (1,0) {};
  \node[vertex] (k) at (0,-1) {};
  \node[vertex] (f) at (0,0) {};
  \draw[->] (i)--(j); \draw[->] (l)--(j); \draw[->] (i)--(k); \draw[->] (l)--(k);
  \draw[->] (f)--(j); \draw[->] (f)--(k);
\end{tikzpicture}}
\newcommand{\QfrDfour}{\begin{tikzpicture}[quiver,scale=0.6]%
  \node[vertex] (c) at (0,0) {};
  \node[vertex] (a) at (-1,1) {};
  \node[vertex] (b) at (-1,-1) {};
  \node[vertex] (d) at (1,1) {};
  \node[vertex] (e) at (1,-1) {};
  \node[vertex] (f) at (1.9,0) {};
  \draw[->] (a)--(c); \draw[->] (b)--(c); \draw[->] (d)--(c); \draw[->] (e)--(c);
  \draw[->] (f)--(c);
\end{tikzpicture}}

\newcommand{\vbinom}[2]{\genfrac{[}{]}{0pt}{}{#1}{#2}}

\usepackage{xpatch}

\makeatletter
\patchcmd\blx@bblinput{\blx@blxinit}
                      {\blx@blxinit
                      }{}{\fail}
\makeatother

\title{Series of Fano quiver moduli from extended Dynkin quivers}
\author{Pieter Belmans \and Markus Reineke}

\begin{document}
\maketitle

\begin{abstract}
  Associated to small extended Dynkin quivers,
  we construct series of Fano varieties of fixed Picard rank at most five and arbitrary even dimension,
  which admit a very explicit description of their geometry.
  For example, they all admit a flat morphism to a projective space,
  whose general fiber is a product of projective lines.
  We also give a closed formula for all Betti numbers,
  and interpret this formula in terms of 
  an action of a cohomological Hall algebra on the cohomology of
  all the varieties in a series.
\end{abstract}

\tableofcontents

\section{Introduction}
We consider extended Dynkin quivers together with multiples of their unique imaginary Schur root,
and study the framed quiver moduli one can attach to this setting.
Our first result is that,
under a natural symmetry condition on the Euler form,
the morphism induced by forgetting the framing
defines a flat map to a lower-dimensional projective space,
making their geometry amenable to an explicit description.

We also describe
when the natural stability parameter for a framed quiver moduli space
is in the same chamber as the canonical stability parameter,
making the framed quiver moduli space a Fano variety.
This setting leads to an interesting and tractable classification problem,
which results in five infinite series
(and some sporadic cases, which we ignore).
By considering all the varieties in an infinite series at once,
we can exhibit closed expressions for their Betti numbers,
and we uncover an action of a cohomological Hall algebra on the cohomology of the varieties
to give a structural interpretation to this pattern for the Betti numbers. The intrinsic reason for this favourable behaviour is that, since our starting point is an extended Dynkin quiver, all relevant semistable quiver Donaldson-Thomas invariants can be computed and the corresponding cohomological Hall algebra can be described by explicit generators.

The starting point for this classification and these results
was the realization that
one of the Fano fourfolds studied in \cite{2607.12895},
a priori defined using a subspace quiver,
also admits an interpretation in terms of a framing of the extended~$\mathrm{D}_4$-quiver.
It turns out that this alternative perspective 
using Dynkin quivers
is the one that explains the interesting fibration structure described in \S5 of op.~cit.

In this paper we greatly extend this geometric description
in dimension~4 for a single quiver
to a general framework
in all even dimensions
for framings of extended Dynkin quivers.
This gives us an interesting class of quiver moduli
with a particularly elegant and explicit geometry.

\paragraph{Setup}
Let~$Q$ be an extended Dynkin quiver,
let~$\delta$ be its unique imaginary Schur root,
and write~$\langle-,-\rangle$ for the Euler form.
We consider the stability parameter given by~$\partial\colonequals\langle\delta,-\rangle$.
We ask that the antisymmetrized Euler form~$\{-,-\}$
vanishes identically on~$\Ker(\partial)$,
which is the Grothendieck group of the subcategory of regular modules.
This already forces~$Q$ to be of type~$\Aone$, $\Atwo$, $\Athree$, or~$\Dfour$.

Let~$\vect{n}\in\mathbb{N}Q_0$ be a framing vector,
and denote~$\widehat{Q}$ the framed quiver,
with framing vertex~$i_0$.
The dimension vector~$k\delta$ for~$k\geq 1$
induces the dimension vector~$\smash{\widehat{k\delta}\colonequals k\delta+\vect{i}_0}$.
Next, choose a positive linear functional~$\lambda\in\Hom(\mathbb{Z}Q_0,\mathbb{Z})$
and~$C=C_k\in\mathbb{N}$ such that~$\lambda(k\delta)<C\gcd(\partial)$,
and define~$\widehat{\partial}\colonequals C\partial-\lambda+\lambda(k\delta)\vect{i}_0^*$.
The choice of~$\lambda$ and~$C$ does not change the notion of stability,
they are all in the same chamber.

Then we define the framed quiver moduli space
\begin{equation}
  \modulispace[\partial\framed]{Q,k\delta,\vect{n}}
  \colonequals
  \modulispace[\widehat{\partial}\semistable]{\widehat{Q},\widehat{k\delta}}
  =
  \modulispace[\widehat{\partial}\stable]{\widehat{Q},\widehat{k\delta}},
\end{equation}
and we have the forgetful morphism
\begin{equation}
  \label{equation:forgetful-introduction}
  p\colon\modulispace[\partial\framed]{Q,k\delta,\vect{n}}
  \to
  \modulispace[\partial\semistable]{Q,k\delta}
  \cong
  \mathbb{P}^k.
\end{equation}
See \cref{section:construction} for more details on the construction and the basic properties.

\paragraph{The forgetful morphism}
The Luna stratification of
the unframed moduli space in \eqref{equation:forgetful-introduction}
encodes how a polystable representation decomposes into stable summands.
The symmetry assumption on the Euler form
makes it possible to obtain the following uniform description of the fibers
of the forgetful morphism.

\begin{alphatheorem}
  \label{theorem:forgetful-morphism}
  For every nonzero framing vector~$\vect{n}$ and every~$k\geq 1$,
  we have the following properties for
  the forgetful morphism~$p$ in \eqref{equation:forgetful-introduction}:
  \begin{enumerate}
    \item it is flat of relative dimension~$(\vect{n}\cdot\delta-1)k$,
    \item its general fiber is~$(\mathbb{P}^{\vect{n}\cdot\delta-1})^k$, and
    \item over every Luna stratum of the base,
      its fibers are products of
      framed nilpotent moduli spaces for the Jordan quiver
      and the oriented two-cycle.
  \end{enumerate}
\end{alphatheorem}
The detailed description of the fibers
is given in \cref{proposition:fibers-forgetful},
and leads to the flatness and description of the general fiber
in \cref{corollary:flat-forgetful}.

In preparation for the special cases studied next,
we explain in \cref{subsection:betti}
how the quiver DT/PT correspondence
can be used to study Betti numbers of these framed moduli spaces.

\paragraph{The classification}
If we ask that the framing stability parameter is
in the same chamber as the canonical stability parameter
(which is defined without using its framed structure),
then the framed moduli space is a Fano variety.
Write~$X_k^{\vect{n}}(Q)=\modulispace[\partial\framed]{Q,k\delta,\vect{n}}$.
This brings us in the setting of \cite{2607.12895}.
It turns out that there is an interesting and tractable classification problem
for this setting,
which leads to the following result,
identifying five interesting series of Fano quiver moduli
constructed from extended Dynkin quivers.

\begin{alphatheorem}
  \label{theorem:5-series}
  The five quiver and framing pairs in \cref{table:5-quivers}
  satisfy~$\vect{n}\cdot\delta=2$ and~$n_i\geq 1$ whenever~$\partial(\vect{i})\leq 0$.

  For each such pair and every~$k\geq 1$,
  the framed moduli space~$X_k^{\vect{n}}(Q)$
  is a smooth Fano variety of dimension~$2k$ and Picard rank~$|Q_0|$.
  It is rational, of pure Hodge--Tate type, and infinitesimally rigid.
  Its index is~$2$ in the case~$\AoneA$ and~$1$ in the other four cases.

  For the even Betti numbers we have that~$\mathrm{b}_{2i}(X_k^{\vect{n}}(Q))=a_{\min\{i,2k-i\}}$, for~$i=1,\ldots,2k$,
  where the~$a_i$ are defined by
  \begin{equation}
    \sum_{i\geq 0}a_ix^i=\frac{(1+x)^{|Q_0|}}{(1-x^2)^3}.
  \end{equation}
\end{alphatheorem}

\begin{table}[t]
  \centering
  \begin{tabular}{ccccc}
    \toprule
    & $Q$ & $\delta$ & $\vect{n}$ & $\widehat{Q}$ \\
    \midrule
    $\AoneA$  & \QAone   & $(1,1)$       & $2\vect{j}$         & \QfrAoneA  \\
    \addlinespace[10pt]
    $\AoneB$  & \QAone   & $(1,1)$       & $\vect{i}+\vect{j}$ & \QfrAoneB  \\
    \addlinespace[10pt]
    $\Atwo$   & \QAtwo   & $(1,1,1)$     & $\vect{j}+\vect{k}$ & \QfrAtwo   \\
    \addlinespace[10pt]
    $\Athree$ & \QAthree & $(1,1,1,1)$   & $\vect{j}+\vect{k}$ & \QfrAthree \\
    \addlinespace[10pt]
    $\Dfour$  & \QDfour  & $(1,1,1,1;2)$ & $\vect{j}$          & \QfrDfour  \\
    \bottomrule
  \end{tabular}
  \caption{The five extended Dynkin quivers and framing data, together with their imaginary roots and framed quivers.}
  \label{table:5-quivers}
\end{table}

The description of the Betti numbers in \cref{theorem:5-series}
is much more explicit than the general results in \cref{subsection:betti}.
There is a conceptual explanation for this nice behavior:
the symmetry of the Euler form gives
a semistable cohomological Hall algebra action on
the cohomology of all varieties in a series.
The resulting cyclic module provides a structural explanation
for considering their Betti numbers together;
see \cref{section:coha} for more details.

In \cref{section:low-dimensional}
we discuss what happens for low-dimensional cases
of the Fano varieties in \cref{theorem:5-series}:
we compute some numerical invariants,
which ought to be useful in comparing our constructions to the literature.
For~$k=1$, the del Pezzo surfaces in the five series
are~$\mathbb{P}^1\times\mathbb{P}^1$
and the blowups of~$\mathbb{P}^2$ at one to four points,
see \cref{subsection:surfaces}.
For~$k=2$, the Fano fourfolds in the five series
come equipped with forgetful maps to~$\mathbb{P}^2$,
which are involution surface bundles with mild degeneration,
see \cref{subsection:fourfolds}.

\paragraph{AI disclosure}
No generative AI was used for the results in this paper.

\paragraph{Acknowledgements}
M.R.~would like to thank T.~Peerenboom for advice on Cohomological Hall algebras.
P.B.~was partially supported by NWO (\href{https://doi.org/10.61686/RZKLF82806}{\texttt{doi:10.61686/RZKLF82806}}).
This work was done during the residency of M.R.~at Utrecht University
as the 2025--2026 Springer Visiting Professor,
and we gratefully acknowledge the support of the Utrecht Geometry Centre.

\section{Recollections on quiver moduli spaces}
We recall the quiver notation and stability conventions used in this paper.
For fuller accounts, see, e.g., \cite{MR2484736,MR3727119};
see also \cite[\S2]{2607.12895} for the same setup in a related paper.

\subsection{Quivers and quiver moduli}
\label{subsection:quivers-and-moduli}
Let~$Q=(Q_0,Q_1)$ be a quiver,
with the set of vertices~$Q_0$ and the set of arrows~$Q_1$.
We consider the~$\mathbb{Z}$-lattice~$\mathbb{Z}Q_0$
spanned by basis elements~$\vect{i}$, indexed by~$i\in Q_0$.
On~$\mathbb{Z}Q_0$ we have the \emph{Euler form}~$\langle-,-\rangle$, defined as
\begin{equation}
  \label{equation:euler-form}
  \langle\vect{d},\vect{e}\rangle\colonequals\sum_{i\in Q_0}d_ie_i-\sum_{\alpha\colon i\to j}d_ie_j.
\end{equation}
Its symmetrization~$(-,-)$ is defined as~$(\vect{d},\vect{e})\colonequals\langle\vect{d},\vect{e}\rangle+\langle\vect{e},\vect{d}\rangle$.
Its antisymmetrization~$\{-,-\}$ is defined as~$\{\vect{d},\vect{e}\}\colonequals\langle\vect{d},\vect{e}\rangle-\langle\vect{e},\vect{d}\rangle$.

We will work over the base field~$\mathbb{C}$ of complex numbers.
We will consider finite-dimensional representations~$V$ of the quiver~$Q$,
which are given by~$\mathbb{C}$-vector spaces~$V_i$ for~$i\in Q_0$,
and~$\mathbb{C}$-linear maps~$V(\alpha)\colon V_i\rightarrow V_j$ for~$\alpha\in Q_1$.
The dimension vector~$\dimvect(V)\in\mathbb{N}Q_0\subset\mathbb{Z}Q_0$
of~$V$ is given as~$\dimvect(V)=\sum_{i\in Q_0}(\dim_\mathbb{C}V_i)\vect{i}$.

\paragraph{Root systems}
For a quiver~$Q$ without loops,
the symmetrized Euler form determines a root system with set of positive roots~$\Delta_+(Q)$.
By Kac's theorem,
these are precisely the dimension vectors of indecomposable representations of~$Q$.
We say that
a positive root~$\alpha$ is real if~$\langle\alpha,\alpha\rangle=1$,
and imaginary otherwise.
For an extended Dynkin quiver,
the positive imaginary roots are the multiples of its indivisible null root~$\delta$;
see \cite[\S7.3]{MR3727119}.

\paragraph{Quiver moduli}
For a dimension vector~$\vect{d}$,
the affine space~$\Rep(Q,\vect{d})$ parametrizes representations of~$Q$ on fixed vector spaces of dimensions~$d_i$,
and~$\mathrm{G}_{\vect{d}}=\prod_{i\in Q_0}\mathrm{GL}_{d_i}$ acts by change of basis.
A stability parameter for~$\vect{d}$ is a~$\Theta\in\mathbb{Z}Q_0^\vee$ for which~$\Theta(\vect{d})=0$.
A representation~$V$ of dimension~$\vect{d}$ is~$\Theta$-semistable
(respectively,~$\Theta$-stable)
if~$\Theta(\dimvect(U))\leq 0$
(respectively, $<0$)
for every nonzero proper subrepresentation~$U\subset V$.
The moduli space~$\modulispace[\Theta\semistable]{Q,\vect{d}}$
parametrizes the corresponding polystable representations,
while~$\modulispace[\Theta\stable]{Q,\vect{d}}$ is its stable locus.

The \emph{canonical stability parameter} for~$\vect{d}$ is
\begin{equation}
  \Theta_\can=\{\vect{d},-\}=\langle\vect{d},-\rangle-\langle-,\vect{d}\rangle.
\end{equation}
The hyperplanes~$\Theta(\vect{e})=0$ for proper nonzero~$\vect{e}<\vect{d}$
divide the space of stability parameters into chambers;
parameters in the same chamber define the same semistable locus.
We call~$\vect{d}$ \emph{$\Theta$-coprime}
if~$\Theta(\vect{e})\neq 0$ for every such~$\vect{e}$.
We call it \emph{$\Theta$-amply stable}
if the complement of the semistable locus in~$\Rep(Q,\vect{d})$ has codimension at least two.
When the stable locus is nonempty,
the stronger numerical condition
\begin{equation}
  \langle\vect{e},\vect{d}-\vect{e}\rangle\leq -2
  \quad\text{whenever}\quad
  0<\vect{e}<\vect{d}\text{ and }\Theta(\vect{e})>0
\end{equation}
implies ample stability by the unstable-locus argument of
\cite[Proposition~5.1]{MR3683503}
and is the criterion we use below.

The following theorem summarizes the properties of quiver moduli in the ideal circumstances.
\begin{theorem}
  \label{theorem:quiver-moduli-properties}
  Assume the following conditions:
  \begin{itemize}[noitemsep]
    \item $Q$ is acyclic,
    \item $\Theta$ does not belong to a proper wall,
    \item $\vect{d}$ is a~$\Theta$-coprime and sincere dimension vector, and
    \item $\vect{d}$ is~$\Theta$-amply stable.
  \end{itemize}
  Then~$\modulispace[\Theta\semistable]{Q,\vect{d}}=\modulispace[\Theta\stable]{Q,\vect{d}}$ is
  a smooth projective variety of dimension~$1-\langle\vect{d},\vect{d}\rangle$,
  and Picard rank~$|\supp(\vect{d})|-1$.
  It is rational, of pure Hodge--Tate type,
  its space of global vector fields is isomorphic to~$\mathrm{HH}^1(\mathbb{C}Q)$,
  and it is infinitesimally rigid.
  If~$\Theta=\Theta_\can$ (or simply in the same chamber),
  then it is a Fano variety,
  of index~$\gcd(\Theta_\can)=\gcd\{\Theta_{\can,i}\mid i\in Q_0\}$.
\end{theorem}
The smoothness follows from the fact that semistability agrees with stability
under the coprimality assumption,
and the projectivity follows from the assumption that~$Q$ is acyclic.
The Picard rank is then given in~\cite[Proposition~3.1]{MR4352662}.
The rationality of coprime quiver moduli is established in~\cite[Theorem~6.4]{MR1914089}.
The pure Hodge--Tate type is~\cite[Theorem~3]{MR1324213}.
The global vector fields isomorphism is~\cite[Theorem~B]{MR4883104}
(with Happel's computation of~$\mathrm{HH}^1(\mathbb{C}Q)$ in~\cite{MR1035222}).
The infinitesimal rigidity corresponds to~\cite[Corollary~D]{MR4954467}.
The Fano property and index are~\cite[Theorem~4.2]{MR4352662}.

An important caveat to make is
that the standing assumption in \cite{MR4883104,MR4954467}
is that of \emph{strong} ample stability
(or as explained in \cite[Remark~1.3]{MR4954467},
it suffices that a certain inequality holds
for every non-trivial Harder--Narasimhan type).
This condition is replaced by the weaker condition
of ample stability in \cite{2610.03302}.

\begin{remark}
  \label{remark:non-coprime}
  We are not always in the ideal situation of \cref{theorem:quiver-moduli-properties}:
  we will also have to consider \emph{multiples} of a Schur root,
  so the coprimality will necessarily fail.
  However, in the very special case of extended Dynkin quivers,
  we will be able to identify these moduli spaces with projective spaces,
  and thus we stay in the world of smooth projective varieties.
\end{remark}

\subsection{Framed quiver moduli}
\label{subsection:framed-quiver-moduli}
The rest of the paper crucially uses framed quiver moduli
(or smooth models),
as introduced in \cite[\S3]{MR2511752}.
We recall this setup, and the most important properties.

\begin{definition}
  \label{definition:framed}
  Given a quiver~$Q$,
  dimension vector~$\vect{d}$,
  a stability parameter~$\Theta\in\mathbb{Z}Q_0^\vee$
  such that~$\Theta(\vect{d})=0$,
  and a framing datum~$\vect{n}\in \mathbb{N}Q_0$,
  we define the \emph{framed quiver setting} as
  \begin{itemize}[noitemsep]
    \item the quiver $\widehat{Q}$, by adding a vertex $i_0$ to $Q_0$,
      together with $n_i$ arrows~$i_0\rightarrow i$ for all~$i\in Q_0$;
    \item the dimension vector $\widehat{\vect{d}}\colonequals\vect{d}+\vect{i}_0\in\mathbb{N}\widehat{Q}_0$;
    \item for a choice of a positive $\lambda\in\mathbb{Z}Q_0^\vee$ and a positive integer $C$
      such that $\lambda(\vect{d})<C\gcd(\Theta)$,
      the stability parameter
      \begin{equation}
        \widehat{\Theta}=\widehat{\Theta}_{\lambda,C}\colonequals C\Theta-\lambda+\lambda(\vect{d})\vect{i}_0^*.
      \end{equation}
  \end{itemize}
  Using this, we define the \emph{framed quiver moduli space} (or \emph{smooth model}) as
  \begin{equation}
    \modulispace[\Theta\framed]{Q,\vect{d},\vect{n}}
    \colonequals
    \modulispace[\widehat{\Theta}\semistable]{\widehat{Q},\widehat{\vect{d}}}
  \end{equation}
\end{definition}
We have the following result,
which also explains why we are not including~$\lambda$ and~$C$ in our notation;
see \cite[Proposition~3.6]{MR2511752}.

\begin{theorem}
  \label{theorem:framed}
  The framed moduli space $\modulispace[\Theta\framed]{Q,\vect{d},\vect{n}}$
  does not depend on the choices for $\lambda$ and~$C$,
  since all resulting stability parameters belong to the same chamber in $\Stab({\widehat{\vect{d}}})$.
  It is always smooth and irreducible, if non-empty,
  of dimension $\vect{n}\cdot\vect{d}-\langle\vect{d},\vect{d}\rangle$,
  and admits a projective map
  \begin{equation}
    \label{equation:forgetful-general}
    p\colon\modulispace[\Theta\framed]{Q,\vect{d},\vect{n}}\to\modulispace[\Theta\semistable]{Q,\vect{d}},
  \end{equation}
  defined by forgetting the framing data.
\end{theorem}

The fibres can be described as follows using so-called local quivers.
Let~$V=\bigoplus_{k=1}^s{U_k}^{\oplus m_k}$ be a~$\Theta$-polystable representation,
where we can assume that the~$U_k$ are chosen to be pairwise non-isomorphic.
Then we call the tuple of dimension vectors
\begin{equation}
  \xi=((\dimvect(U_k))^{m_k})_{k=1}^s
\end{equation}
the \emph{Luna type} of~$V$.
The set~$S_\xi\subset\modulispace[\Theta\semistable]{Q,\vect{d}}$ of representations
of fixed Luna type~$\xi$ is locally closed,
and the quotient map by the group ${\rm GL}(\vect{d})$
\begin{equation}
  \Rep^{\Theta\semistable}(Q,\vect{d})\rightarrow\modulispace[\Theta\semistable]{Q,\vect{d}}
\end{equation}
is \'etale-locally trivial over~$S_\xi$.
We call~$S_\xi$ the \emph{Luna stratum} in~$\modulispace[\Theta\semistable]{Q,\vect{d}}$
associated to the Luna type~$\xi$.

We then define the \emph{local quiver}~$Q_\xi$ for~$\xi$
as the quiver with vertices~$i_1,\ldots,i_s$,
and with the number of arrows from~$i_k$ to~$i_l$ given by
\begin{equation}
  \delta_{k,l}-\langle\dimvect(U_k),\dimvect(U_l)\rangle_Q.
\end{equation}
We also define the local dimension vector~$\vect{d}_\xi=\sum_{k=1}^s m_k\vect{i}_k$ for~$Q_\xi$, and
we consider the local framing datum $\vect{n}_\xi=\sum_{k=1}^s(\vect{n}\cdot\dimvect(U_k))\vect{i}_k$.
The following result is \cite[Theorem~4.1]{MR2511752}.

\begin{proposition}
  \label{proposition:fibres-framed}
  Consider the setup of \cref{theorem:framed},
  and let~$\xi$ be a Luna type of~$\modulispace[\Theta\semistable]{Q,\vect{d}}$.
  For any~$V\in S_\xi$, the fiber of the forgetful map~$p$ from \eqref{equation:forgetful-general} is
  \begin{equation}
    p^{-1}(V)\cong\modulispace[0\framed,\nilpotent]{Q_\xi,\vect{d}_\xi,\vect{n}_\xi}.
  \end{equation}
\end{proposition}

A representation of the latter moduli space consists of
a nilpotent representation $V$ of $Q_\xi$ of dimension vector $\vect{d}_\xi$,
together with linear maps $f_k\colon\mathbb{C}^{\vect{n}\cdot\vect{d}_k}\rightarrow V_{i_k}$ for all $k$.
Such a tuple is stable if and only if the images of the $f_k$ generate $V$ as a representation of $Q_\xi$.

\section{The construction}
\label{section:construction}
We now specialize the preceding framework to multiples of the null root of an extended Dynkin quiver.
First we identify the orientations for which the Euler form is symmetric on the defect-zero lattice
and describe their stable dimension vectors.
We then use framed moduli spaces and local quivers to analyze the forgetful morphism,
before applying the quiver DT/PT correspondence to their cohomology.

\subsection{Extended Dynkin quivers}
\label{subsection:extended-dynkin}
We recall and prove some basic results about extended Dynkin quivers
which are necessary for the construction in \cref{subsection:framed}.

Let $Q$ be an acyclic quiver of extended Dynkin (or Euclidean) type,
so the underlying unoriented graph is an extended Dynkin diagram of type
\begin{itemize}
  \item $\tildeA_n$ for~$n\geq 1$,
  \item $\tildeD_n$ for~$n\geq 4$, or
  \item $\tildeE_n$ for~$n=6,7,8$.
\end{itemize}
Because~$Q$ is of extended Dynkin type,
the quadratic form attached to the symmetrization~$(-,-)$ of the Euler form
is semi-positive definite,
with radical generated by an indivisible vector~$\delta\in\mathbb{N}Q_0$.
We thus have~$(\delta,-)=0$, and~$\langle\delta,\delta\rangle=0$.

Following Dlab and Ringel \cite{MR0447344}
(see also \cite[\S3.5]{MR0774589}),
we define the \emph{defect function} as~$\partial=\langle\delta,-\rangle$,
thus we have~$\{\delta,-\}=2\partial$.
Let~$\Ker(\partial)$ be the subgroup of $\mathbb{Z}Q_0$
of dimension vectors $\vect{d}$ such that $\partial(\vect{d})=0$,
and define
\begin{equation}
  \label{equation:P}
  P\colonequals\{\vect{d}\in\Ker(\partial)\cap\Delta_+\mid\vect{d}\lneqq\delta\},
\end{equation}
the set of positive roots in~$\Ker(\partial)$ which are smaller than~$\delta$.
For more on the representation theory of~$\mathbb{C}Q$,
a tame and hereditary algebra,
we refer to \cite[\S3.6]{MR0774589}.
For instance, the following lemma is a consequence of the trichotomy
for indecomposable representations of~$\mathbb{C}Q$.

\begin{lemma}
  \label{lemma:stables}
  Let~$Q$ be an extended Dynkin quiver.
  If a dimension vector~$\vect{e}\in\Ker(\partial)$ admits a $\partial$-stable representation,
  then~$\vect{e}=\delta$ or~$\vect{e}\in P$.
  In the former case, we have a~$\mathbb{P}^1$ worth of semistable representations, and a dense open subset of stable representations.
  In the latter case, there exists a unique (semi-)stable such representation up to isomorphism.
\end{lemma}

\begin{proof}
  A stable representation is in particular indecomposable.
  Thus by the trichotomy it is either preprojective or preinjective
  (which implies that the defect of its dimension vector is non-zero, which we exclude)
  or it is a regular indecomposable.
  In the latter case,
  its dimension vector is a positive root in the kernel of~$\partial$,
  and is at most~$\delta$.
  If it is equal to~$\delta$,
  then it is an imaginary root,
  we have a~$\mathbb{P}^1$ worth of semistable representations, and a dense open subset of stable representations.
  If it is strictly below~$\delta$,
  it is a real root,
  and a stable representation is unique up to isomorphism.
\end{proof}

We need the following fact first proven in \cite[Proposition~2.14]{MR5085658}.
We provide a proof in our current setup,
without reference to weighted projective lines.

\begin{lemma}
  \label{lemma:4-quivers}
  Let~$Q$ be an extended Dynkin quiver.
  If the restriction of $\{-,-\}$ to $\Ker(\partial)$ vanishes identically,
  then $Q$ is already of type $\Aone$, $\Atwo$, $\Athree$ or $\Dfour$.
\end{lemma}

\begin{proof}
  The condition on the antisymmetrization
  translates to the symmetry of the Euler form on the category of regular modules
  (whose Grothendieck group is~$\Ker(\partial)$ by the trichotomy),
  which implies that all tubes in the Auslander--Reiten quiver have width at most two.
  A direct inspection of the tubular types
  on \cite[page~158]{MR0774589} immediately yields the claim.
\end{proof}

\begin{lemma}
  \label{lemma:euler-form-stables}
  Let~$Q$ be an extended Dynkin quiver.
  If the restriction of $\{-,-\}$ to $\Ker(\partial)$ vanishes identically,
  then
  \begin{equation}
    P=\{\alpha_1,\ldots,\alpha_{t},\overline{\alpha}_1,\ldots,\overline{\alpha}_{t}\}
  \end{equation}
  for pairwise distinct roots~$\alpha_1,\ldots,\alpha_{t}$,
  where~$t=\#Q_0-2$,
  and~$\overline{\alpha}_i\colonequals\delta-\alpha_i$.
  For all $i,j=1,\ldots,t$ with~$i\neq j$, we have
  \begin{equation}
    \langle\alpha_i,\alpha_j\rangle=\langle\overline{\alpha}_i,\overline{\alpha}_j\rangle=\langle\alpha_i,\overline{\alpha}_j\rangle=0
  \end{equation}
  and
  \begin{equation}
    \begin{gathered}
      \langle\alpha_i,\alpha_i\rangle=\langle\overline{\alpha}_i,\overline{\alpha}_i\rangle=1, \\
      \langle\alpha_i,\overline{\alpha}_i\rangle=\langle\overline{\alpha}_i,\alpha_i\rangle=-1.
    \end{gathered}
  \end{equation}
  Moreover, if $\sum_{\alpha\in P}k_\alpha\alpha$ with~$k_\alpha\geq 0$ is a multiple of $\delta$,
  then $k_{\alpha_p}=k_{\overline{\alpha}_p}$ for all $p$.
\end{lemma}

\begin{proof}
  This follows from a direct inspection,
  using the description of regular indecomposable representations
  used in the proof of \cref{lemma:stables}.
\end{proof}

From now on, we will assume that $\{-,-\}$ vanishes identically on  $\Ker(\partial)$,
so that~$Q$ is one of the quivers in \cref{lemma:4-quivers}.
In \cref{table:extended-dynkin}
we describe the relevant properties of these 4 Dynkin types.

\begin{table}
  \centering
  \begin{tabular}{cc}
    \toprule
    extended Dynkin type & unique indivisible null root $\delta$  \\
    \midrule
    $\Aone$   & $(1,1)$ \\
    $\Atwo$   & $(1,1,1)$ \\
    $\Athree$ & $(1,1,1,1)$ \\
    $\Dfour$  & $(1,1,1,1;2)$ \\
    \bottomrule
  \end{tabular}
  \caption{Overview of the extended Dynkin quivers from \cref{lemma:4-quivers}.}
  \label{table:extended-dynkin}
\end{table}

\subsection{Framing quiver moduli of extended Dynkin quivers}
\label{subsection:framed}
Moduli spaces for extended Dynkin quivers are not very interesting,
as extended Dynkin quivers are tame.
Thus we will apply the framing construction of \cref{subsection:framed-quiver-moduli}
to construct interesting moduli spaces out of extended Dynkin quivers
and their imaginary roots~$k\delta$.

Our main objects of interest are the framed moduli spaces
\begin{equation}
  X_{\vect{d}}^{\vect{n}}(Q)\colonequals\modulispace[\partial\framed]{Q,\vect{d},\vect{n}}
\end{equation}
for $\vect{d}\in\Ker(\partial)$, and in particular we will consider
\begin{equation}
  X_k^{\vect{n}}(Q)\colonequals\modulispace[\partial\framed]{Q,k\delta,\vect{n}}
\end{equation}
for $k\geq 1$ in what follows.
By \cref{theorem:framed},
we have that~$X_k^{\vect{n}}(Q)$ is
an irreducible smooth projective variety of dimension~$\vect{n}\cdot k\delta-\langle k\delta,k\delta\rangle=(\vect{n}\cdot\delta)k$.

\subsection{The fibers of the forgetful morphism}
\label{subsection:fibers}
The natural map \eqref{equation:forgetful-general}
from the framed to the unframed moduli space specializes to a map
\begin{equation}
  \label{equation:forgetful}
  p_k\colon X_k^{\vect{n}}(Q)\to \modulispace[\partial\semistable]{Q,k\delta}\cong\Sym^k\mathbb{P}^1\cong\mathbb{P}^k.
\end{equation}
The identification with~$\Sym^k\mathbb{P}^1\cong\mathbb{P}^k$
uses that every~$\partial$-polystable
is in fact the direct sum of~$k$~stable representations,
which are parameterized by~$\mathbb{P}^1$,
see also \cite[Corollary~7.3]{MR1900290}.

By \cref{proposition:fibres-framed}
we need to first determine the Luna stratification of the base.
We use the notation of \cref{lemma:euler-form-stables}.
\begin{lemma}
  \label{lemma:luna-types}
  The Luna types of~$\modulispace[\partial\semistable]{Q,k\delta}$
  are of the form
  \begin{equation}
    \xi=(\delta^{k_1},\ldots,\delta^{k_s},\alpha_1^{\ell_1},\ldots,\alpha_t^{\ell_t},\overline{\alpha}_1^{\ell_1},\ldots,\overline{\alpha}_{t}^{\ell_t}),
  \end{equation}
  where~$k_1+\ldots+k_s+\ell_1+\ldots+\ell_t=k$.
  If~$\xi$ is one of these Luna types,
  then the dimension of the corresponding Luna stratum~$S_\xi$ is equal to~$s$.
\end{lemma}

\begin{proof}
  By \cref{lemma:stables}, the dimension vectors~$\alpha\in P$ and~$\delta$
  are the only ones admitting~$\partial$-stable representations.
  For every~$\alpha\in P$, there is a unique stable representation of that dimension vector.
  Hence,
  for~$\alpha\in P$
  at most one~$\alpha^{\ell_\alpha}$ can appear in the Luna type.
  We formally set~$\ell_\alpha=0$ in case~$\alpha$ does not appear.
  That~$\alpha_p$ and~$\overline{\alpha}_p$ appear with the same multiplicity
  follows from the last statement in \cref{lemma:euler-form-stables}.
  For~$\delta$ there is a one-dimensional family of stable representations,
  so every~$k_p$ contributes one to the dimension of the Luna stratum.
\end{proof}

The next step in applying \cref{proposition:fibres-framed}
is determining the local quiver settings.
We will consider the Jordan quiver
\begin{equation}
  \mathrm{J}\colon
  \begin{tikzpicture}[quiver, baseline=-0.5ex]
    \node[vertex] (i) at (0,0) {};
    \draw[->] (i) to[loop right] (i);
  \end{tikzpicture},
\end{equation}
and the oriented two-cycle quiver
\begin{equation}
  \mathrm{C}_2\colon
  \begin{tikzpicture}[quiver, baseline=-0.5ex]
    \node[vertex] (i) at (0,0) {};
    \node[vertex] (j) at (1,0) {};
    \draw[->] (i) to[bend left=16] (j);
    \draw[->] (j) to[bend left=16] (i);
  \end{tikzpicture}.
\end{equation}
With this notation, we have the following.
\begin{lemma}
  \label{lemma:local-quivers}
  Let~$\xi$ be a Luna type
  as determined in \cref{lemma:luna-types}.
  Then the local quiver~$Q_\xi$ is
  the disjoint union of
  \begin{itemize}[noitemsep]
    \item $s$ copies of the Jordan quiver~$\mathrm{J}$,
    \item $t$ copies of the oriented two-cycle quiver~$\mathrm{C}_2$.
  \end{itemize}
  The local dimension vector~$\vect{d}_\xi$ is
  \begin{equation}
    \vect{d}_\xi=\sum_{p=1}^sk_p\vect{i}_p+\sum_{p=1}^{t}\ell_p(\vect{j}_p+\vect{j}_p'),
  \end{equation}
  with the convention that~$\ell_\alpha=0$ if~$\alpha\in P$ does not appear in the Luna type.
  The local framing datum~$\vect{n}_\xi$ is
  \begin{equation}
    \vect{n}_\xi=(\vect{n}\cdot\delta)\sum_{p=1}^s\vect{i}_p+\sum_{p=1}^t(\vect{n}\cdot\alpha_p)\vect{j}_p+\sum_{p=1}^t(\vect{n}\cdot\overline{\alpha}_p)\vect{j}_p'.
  \end{equation}
\end{lemma}

\begin{proof}
  We use \cref{lemma:euler-form-stables} to determine the local quiver~$Q_\xi$.
  It has vertices~$i_1,\ldots,i_s$
  and~$j_1, j_1',\ldots,j_t,j_t'$.
  To every vertex~$i_p$, there is a single loop attached,
  as~$\langle\delta,\delta\rangle=0$,
  whereas there is a single arrow in either direction between each pair of vertices $j_p$, $j_p'$,
  as~$\langle\alpha_j,\alpha_j\rangle=1$
  and~$\langle\alpha_j,\overline{\alpha}_j\rangle=-1$
  (and other pairings are equal to~0).
  The local dimension vector and local framing datum follow.
\end{proof}

We have the following description of the fibers of the forgetful morphism.
\begin{proposition}
  \label{proposition:fibers-forgetful}
  The fiber of the map~$p_k$ from \eqref{equation:forgetful}
  in a point~$V$ of the Luna stratum~$S_\xi$ indexed by
  a Luna type~$\xi$ from \cref{lemma:luna-types}
  is given by
  \begin{equation}
    \label{equation:fiber}
    p_ k^{-1}(V)\cong
    \prod_{p=1}^s\modulispace[0\framed,\nilpotent]{\mathrm{J},k_p,\vect{n}\cdot\delta}
    \times
    \prod_{p=1}^{t}\modulispace[0\framed,\nilpotent]{\mathrm{C}_2,(\ell_p,\ell_p),(\vect{n}\cdot\alpha_{p},\vect{n}\cdot\overline{\alpha}_{p})}.
  \end{equation}
\end{proposition}

\begin{proof}
  We apply \cref{proposition:fibres-framed}.
  Using the decomposition into connected components of the quiver $Q_\xi$,
  the fibre $p_k^{-1}(V)$ decomposes into the product of
  the corresponding framed nilpotent moduli spaces of the loop (resp.~two-cycle) quiver.
\end{proof}

The factors in \eqref{equation:fiber} are parametrizing the following linear algebra data:
\begin{itemize}
  \item A Jordan factor parametrizes
    a nilpotent endomorphism~$A$ of a vector space~$V$ of dimension~$k_p$,
    together with~$\vect{n}\cdot\delta$ framing vectors that generate~$V$ as a~$\mathbb{C}[A]$-module.
  \item A two-cycle factor parametrizes vector spaces~$V$ and~$W$ of dimension~$\ell_p$,
    morphisms~$A\colon V\to W$ and~$B\colon W\to V$ such that~$A\circ B$ and~$B\circ A$ are nilpotent,
    together with framing maps~$\mathbb{C}^{\vect{n}\cdot\alpha_p}\longrightarrow V$
    and~$\mathbb{C}^{\vect{n}\cdot\overline{\alpha}_p}\longrightarrow W$
    whose images generate the representation.
\end{itemize}
This is a generalization of the descriptions considered in \cite[\S5.2]{2607.12895} for~$k=2$.
We will revisit these in \cref{subsection:fourfolds}.

For now, the following result suffices.
\begin{corollary}
  \label{corollary:flat-forgetful}
  The map~$p_k$ from \eqref{equation:forgetful}
  is equidimensional
  of dimension~$(\vect{n}\cdot\delta-1)k$,
  with general fibre~$(\mathbb{P}^{\vect{n}\cdot\delta-1})^k$.
  In particular, it is flat.
\end{corollary}

\begin{proof}
  By construction, the dimension of $\modulispace[0\framed,\nilpotent]{Q,\vect{d},\vect{n}}$,
  the framed moduli space of nilpotent representations,
  equals $\vect{n}\cdot\vect{d}+\dim\mathrm{N}(Q,\vect{d})-\dim {\rm GL}(\vect{d})$,
  where $\mathrm{N}(Q,d)\subset{\rm Rep}(Q,\vect{d})$ denotes the variety of nilpotent representations (the zero fibre of the quotient map).
  In the case of a symmetric quiver, \cite[Remark~3.6]{MR5038724} gives a dimension formula for $\mathrm{N}(Q,\vect{d})$.
  For the Jordan quiver with dimension~$m$, this formula gives
  \begin{equation}
    \dim\mathrm{N}(\mathrm{J},m)=2\binom{m}{2}=m^2-m.
  \end{equation}
  Since~$\dim{\rm GL}_m=m^2$,
  its factor in \eqref{equation:fiber} has dimension~$(\vect{n}\cdot\delta-1)m$.
  For the oriented two-cycle with dimension vector~$(m,m)$, the same formula gives
  \begin{equation}
    \dim\mathrm{N}(\mathrm{C}_2,(m,m))
    =2\binom{m}{2}+m^2
    =2m^2-m.
  \end{equation}
  Its framing dimension is~$((\vect{n}\cdot\alpha_p)+(\vect{n}\cdot\overline{\alpha}_p))m=(\vect{n}\cdot\delta)m$,
  while the base-change group has dimension~$2m^2$.
  Thus this factor also has dimension~$(\vect{n}\cdot\delta-1)m$,
  proving that the fibre has dimension $(\vect{n}\cdot\delta-1)k$.
  The Zariski-open Luna stratum
  corresponds to the Luna type~$(\delta^1,\ldots,\delta^1)$,
  which describes the general fiber.
  The flatness follows from miracle flatness.
\end{proof}

\subsection{Betti numbers}
\label{subsection:betti}
For an irreducible variety~$X$, we consider the following shifted version of the Poincar\'e polynomial in compactly supported intersection cohomology
\begin{equation}
  P_X(q)=(-q^{1/2})^{-\dim X}\sum_i\dim{\rm IH}_{\mathrm{c}}^i(X,\mathbb{Q})(-q^{1/2})^i\in\mathbb{Z}[q^{\pm 1/2}];
\end{equation}
and we also define $P_\emptyset(q)=0$.
The so-called quiver DT/PT correspondence
relates the shifted Poincar\'e polynomials of framed moduli spaces and unframed moduli spaces.
See \cite{MR4000572} for the original reference,
and \cite[\S2.4]{MR4283580} for the setup we will use.

Denote by~$\overline{\modulispace[\Theta\stable]{Q,\vect{d}}}$
the closure of~$\modulispace[\Theta\stable]{Q,\vect{d}}$
inside~$\modulispace[\Theta\semistable]{Q,\vect{d}}$.
In other words,
\begin{equation}
  \overline{\modulispace[\Theta\stable]{Q,\vect{d}}}
  =
  \begin{cases}
    \modulispace[\Theta\semistable]{Q,\vect{d}} & \modulispace[\Theta\stable]{Q,\vect{d}}\neq\emptyset \\
    \emptyset & \modulispace[\Theta\stable]{Q,\vect{d}}=\emptyset.
  \end{cases}
\end{equation}
The following theorem gives the version of the quiver DT/PT correspondence
proved in \cite[\S2.4]{MR4283580}.

\begin{theorem}
  \label{theorem:dt-pt}
  Assume that the Euler form of $Q$ is symmetric on $\Ker(\Theta)$, and define $\Lambda=\Ker(\Theta)\cap\mathbb{N}Q_0$.
  \begin{equation}
    \sum_{\vect{d}\in\Lambda}(-1)^{\vect{n}\cdot\vect{d}}
    P_{\modulispace[\Theta\framed]{Q,\vect{d},\vect{n}}}(q)x^{\vect{d}}
    =\Exp\left(
    \sum_{0\neq\vect{d}\in\Lambda}(-1)^{\vect{n}\cdot\vect{d}}
    P_{\mathbb{P}^{\vect{n}\cdot\vect{d}-1}}(q)
    P_{\overline{\modulispace[\Theta\stable]{Q,\vect{d}}}}(q)x^{\vect{d}}
    \right).
  \end{equation}
\end{theorem}

Here $\Exp$ denotes the plethystic exponential, which is defined by
\begin{equation}
  \Exp\left(\sum_{0\neq\vect{d}\in\Lambda}\sum_{i\in\frac12\mathbb{Z}}c_{\vect{d},i}q^ix^{\vect{d}}\right)
  =\prod_{0\neq\vect{d}\in\Lambda}\prod_{i\in\frac12\mathbb{Z}}(1-q^ix^{\vect{d}})^{-c_{\vect{d},i}}.
\end{equation}

Applying the quiver DT/PT-correspondence to the framed moduli spaces $X_k^{\vect{n}}(Q)$  of the previous section and using the description of moduli spaces of stable representations in \cref{lemma:stables}, we find

\begin{corollary}
  \label{corollary:dtpt}
  We have the following identity in $\mathbb{Z}[q^{\pm 1/2}][\![\Lambda]\!]$
  \begin{equation}
    \label{equation:dtpt}
    \sum_{\vect{d}\in\Lambda}(-1)^{\vect{n}\cdot\vect{d}}P_{X_{\vect{d}}^{\vect{n}}(Q)}(q)x^{\vect{d}}
    =
    \Exp\left(
    \sum_{\alpha\in P}(-1)^{\vect{n}\cdot\alpha}
    P_{\mathbb{P}^{\vect{n}\cdot\alpha-1}}(q)x^{\alpha}
    +(-1)^{\vect{n}\cdot\delta}
    P_{\mathbb{P}^{\vect{n}\cdot\delta-1}}(q)P_{\mathbb{P}^1}(q)x^{\delta}
    \right).
  \end{equation}
\end{corollary}

This allows us to state an explicit formula for
the shifted Poincar\'e polynomial of $X_k^{\vect{n}}(Q)$
as a sum of products of $v$-binomial coefficients,
where the $v$-binomial coefficient is defined as
\begin{equation}
  \vbinom{n}{k}_v=\prod_{i=1}^k\frac{v^{n+1-i}-v^{-n-1+i}}{v^i-v^{-i}}.
\end{equation}

\begin{corollary}
  \label{corollary:dtpt-explicit}
  For all~$k\geq 1$, we have
  \begin{equation}
    P_{X_k^{\vect{n}}(Q)}(q)
    =
    \sum_{k',k'',(k_p)_p}q^{(k'-k'')/2}
    \vbinom{-\vect{n}\cdot\delta}{k'}_{-q^{1/2}}
    \vbinom{-\vect{n}\cdot\delta}{k''}_{-q^{1/2}}
    \prod_{p=1}^t
    \vbinom{\vect{n}\cdot\alpha_p}{k_p}_{-q^{1/2}}
    \vbinom{\vect{n}\cdot\overline{\alpha}_p}{k_p}_{-q^{1/2}},
  \end{equation}
  where the sum ranges over all nonnegative $k',k'',k_1,\ldots,k_t$ summing up to $k$.
\end{corollary}

\begin{proof}
  From the $q$-binomial theorem, we derive the following identities for~$c\in\mathbb{N}$:
  \begin{equation}
    \begin{aligned}
      \Exp((-1)^cP_{\mathbb{P}^{c-1}}(q)z)
      &=\sum_{k=0}^c(-1)^{ck}\vbinom{c}{k}_{-q^{1/2}}z^k, \\
      \Exp((-1)^cP_{\mathbb{P}^{c-1}}(q)P_{\mathbb{P}^1}(q)z)
      &=\sum_{k\geq 0}(-1)^{ck}
      \sum_{k'+k''=k}q^{(k'-k'')/2}
      \vbinom{-c}{k'}_{-q^{1/2}}
      \vbinom{-c}{k''}_{-q^{1/2}}z^k.
    \end{aligned}
  \end{equation}
  Using these identities and the multiplicativity of~$\Exp$,
  we can express the right hand side of \eqref{equation:dtpt} as a sum over~$\Lambda$
  and compare the coefficient of~$x^{k\delta}$ on both sides, which yields
  \begin{equation}
    P_{X_k^{\vect{n}}(Q)}(q)
    =\sum_{(k_\alpha)_{\alpha\in P},\,k',k''}
    q^{(k'-k'')/2}
    \vbinom{-\vect{n}\cdot\delta}{k'}_{-q^{1/2}}
    \vbinom{-\vect{n}\cdot\delta}{k''}_{-q^{1/2}}
    \prod_{\alpha\in P}\vbinom{\vect{n}\cdot\alpha}{k_\alpha}_{-q^{1/2}},
  \end{equation}
  where the sum ranges over all  $0\leq k_\alpha\leq \vect{n}\cdot\alpha$ for $\alpha\in P$ and $k',k''\geq 0$ such that
  \begin{equation}
    \sum_{\alpha\in P}k_\alpha\alpha+(k'+k'')\delta=k\delta.
  \end{equation}
  By \cref{lemma:euler-form-stables} we know that this is only possible if $k_{\alpha_p}=k_{\overline{\alpha}_p}=:k_p$ for all $p=1,\ldots,t$, as well as $\sum_pk_p+k'+k''=k$. The claimed formula follows.
\end{proof}

\section{Identifying the Fano series}
We now ask which of the framed moduli spaces constructed above
are Fano for every~$k$.
For a framed quiver moduli space,
this can be checked by comparing its framing stability with the canonical stability of the framed quiver.
The resulting numerical conditions leave the five quiver and framing pairs
displayed in \cref{table:5-quivers}.
We then specialize the Betti-number formula from \cref{subsection:betti} to these pairs.

\subsection{Fanoness of framed extended Dynkin quiver moduli}

\begin{theorem}
  \label{theorem:fano-series}
  Assume that~$\vect{n}\cdot\delta=2$
  and~$n_i\geq1$ whenever~$\partial(\vect{i})\leq0$.
  Then, for every~$k\geq1$,
  the framed moduli space~$X_k^{\vect{n}}(Q)$
  is a Fano variety of Picard rank~$|Q_0|$,
  index~$\gcd(2k,n_i : i\in Q_0)$,
  and dimension~$2k$.
  It is rational, of pure Hodge--Tate type, and infinitesimally rigid.
\end{theorem}

\begin{proof}
  By \cref{definition:framed},
  the space~$X_k^{\vect{n}}(Q)$ is a moduli space for~$\widehat{Q}$
  with dimension vector~$\widehat{\vect{d}}=k\delta+\vect{i}_0$
  and stability parameter
  \begin{equation}
    \widehat{\Theta}=C\partial-\lambda+k\lambda(\delta)\vect{i}_0^*,
  \end{equation}
  where~$\lambda\in\mathbb{Z}Q_0^\vee$ is positive
  and~$k\lambda(\delta)<C\gcd(\partial)$.
  Direct inspection of the four relevant quiver types gives~$\gcd(\partial)=1$.
  The canonical stability of~$\widehat{\vect{d}}$ is
  \begin{equation}
    \Theta_\can=2k\partial-\vect{n}+k(\vect{n}\cdot\delta)\vect{i}_0^*.
  \end{equation}
  Comparing coefficients at the vertices of~$Q$ and at the framing vertex,
  the equality~$\widehat{\Theta}=\alpha\Theta_\can$ for~$\alpha>0$ amounts to
  \begin{equation}
    \begin{aligned}
      C\partial-\lambda &= \alpha(2k\partial-\vect{n}),\\
      k\lambda(\delta) &= k\alpha(\vect{n}\cdot\delta).
    \end{aligned}
  \end{equation}
  Evaluating the first equality at~$\delta$ gives the second,
  since~$\partial(\delta)=0$.
  Choose a positive integer~$\alpha$ with~$\alpha>-\partial(\vect{i})$ for every~$i\in Q_0$,
  and set~$C=2k\alpha+1$ and~$\lambda=\partial+\alpha\vect{n}$.
  Then the first equality holds.
  The assumptions on~$\vect{n}$ ensure that~$\lambda$ is positive,
  and~$\vect{n}\cdot\delta=2$ gives~$k\lambda(\delta)=2k\alpha<C$.
  Thus~$\widehat{\Theta}=\alpha\Theta_\can$ with an admissible choice of~$C$ and~$\lambda$.

  We verify strong ample stability.
  If~$0<\vect{e}\leq k\delta$ and~$\widehat{\Theta}(\vect{e})>0$,
  then~$2k\partial(\vect{e})>\vect{n}\cdot\vect{e}$,
  so~$\partial(\vect{e})>0$ and~$\langle\vect{e},\vect{e}\rangle_Q\geq 1$.
  Since~$\langle\vect{e},\delta\rangle=-\partial(\vect{e})$,
  \begin{equation}
    \langle\vect{e},k\delta-\vect{e}+\vect{i}_0\rangle_{\widehat{Q}}
    =-k\partial(\vect{e})-\langle\vect{e},\vect{e}\rangle_Q
    \leq-k-1\leq-2.
  \end{equation}
  For a subdimension vector~$\vect{e}+\vect{i}_0<\widehat{\vect{d}}$
  with~$\widehat{\Theta}(\vect{e}+\vect{i}_0)>0$,
  put~$\vect{f}=k\delta-\vect{e}$.
  Then~$\vect{f}\neq0$ and~$\widehat{\Theta}(\vect{f})<0$,
  while
  \begin{equation}
    \langle\vect{e}+\vect{i}_0,\vect{f}\rangle_{\widehat{Q}}
    =k\partial(\vect{f})-\langle\vect{f},\vect{f}\rangle_Q-\vect{n}\cdot\vect{f}.
  \end{equation}
  If~$\partial(\vect{f})\neq0$,
  the right-hand side is at most~$-2$:
  for positive defect use~$\vect{n}\cdot\vect{f}>2k\partial(\vect{f})$,
  and for negative defect use~$\langle\vect{f},\vect{f}\rangle_Q\geq1$.
  If~$\partial(\vect{f})=0$,
  then~$\vect{n}\cdot\vect{f}>0$;
  either~$\langle\vect{f},\vect{f}\rangle_Q\geq1$,
  or~$\vect{f}$ is a positive multiple of~$\delta$ and~$\vect{n}\cdot\vect{f}\geq2$.
  This proves the same bound.
  The Fano property and the stated invariants now follow from \cref{theorem:quiver-moduli-properties}.
\end{proof}

\begin{remark}
  With some extra effort, one can in fact prove that the conditions in the theorem are necessary for the equations and inequalities in the proof to be fulfillable for all~$k\geq 1$.
\end{remark}

\begin{proposition}
  \label{proposition:5-series}
  The conditions~$\vect{n}\cdot\delta=2$ and~$n_i\geq 1$ whenever~$\partial(\vect{i})\leq 0$
  hold precisely for the quivers and framing data in \cref{table:5-quivers}.
  Consequently, the index of $X_k^{\vect{n}}(Q)$ is one, except in the first case, where it is two. We also see that $\dim{\rm Aut}(X_k^{\vect{n}}(Q))^0=10-2|Q_0|$.
\end{proposition}

\begin{proof} This again follows by direct inspection of all quivers of the relevant extended Dynkin types.
The geometric properties follow from \cref{theorem:quiver-moduli-properties}.
\end{proof}

\subsection{Betti numbers of the Fano series}
A surprising feature for these quiver moduli
is that we can write down
a closed formula for the Betti numbers $\mathrm{b}_{2i}(X_k^{\vect{n}}(Q))$.
This will use the quiver DT/PT correspondence,
already used in \cref{corollary:dtpt,corollary:dtpt-explicit}.
The cohomological Hall algebra action in the next section gives a structural reason
to consider the cohomology of all members of a series together.

\begin{theorem}
  \label{theorem:betti}
  In any of the five cases of \cref{proposition:5-series},
  we have~$\mathrm{b}_{2i}(X_k^{\vect{n}}(Q))=a_i$ for~$i\leq k$, where~$a_i$ is given by
  \begin{equation}
    \sum_{i\geq 0}a_ix^i=\frac{(1+x)^{|Q_0|}}{(1-x^2)^3}.
  \end{equation}
  Explicitly,
  \begin{enumerate}
    \item for $Q$ of type $\tildeA_1$, we have $a_i=\lfloor(i+2)^2/4\rfloor$,
    \item for $Q$ of type $\tildeA_2$, we have $a_i=(i+1)(i+2)/2$,
    \item for $Q$ of type $\tildeA_3$, we have $a_i=(i+1)^2$,
    \item for $Q$ of type $\tildeD_4$, we have $a_i=2i(i+1)+1$.
  \end{enumerate}
\end{theorem}

\begin{proof}
  We specialize \cref{corollary:dtpt} to the present situation where~$\vect{n}\cdot\delta=2$
  and, again by direct inspection of all five cases, $\vect{n}\cdot\alpha=1$ for all~$\alpha\in P$.
  This yields
  \begin{equation}
    \label{equation:dtpt-fano}
    \sum_{\vect{d}\in\Lambda}(-1)^{\vect{n}\cdot\vect{d}}P_{X_{\vect{d}}^{\vect{n}}(Q)}(q)x^{\vect{d}}
    =
    \Exp\left(-\sum_{\alpha\in P}x^{\alpha}+P_{\mathbb{P}^1}(q)^2x^{\delta}\right)
    =\frac{\prod_{\alpha\in P}(1+x^\alpha)}
    {(1-qx^\delta)(1-x^\delta)^2(1-q^{-1}x^\delta)}.
  \end{equation}
  
  Comparing coefficients of powers of~$x^\delta$ on both sides as in the proof of \cref{corollary:dtpt-explicit}
  and noting that~$t=|Q_0|-2$, we find
  \begin{equation}\label{eqn:poincareseries}
    \sum_{k\geq 0} P_{X_{k}^{\vect{n}}(Q)}(q)x^k
    =\frac{(1+x)^{|Q_0|-2}}{(1-qx)(1-x)^2(1-q^{-1}x)}.
  \end{equation}
  
  Denoting by $F(x)$ the series in the claim of the theorem, we can thus calculate
  \begin{equation}
    \begin{aligned}
      \sum_{k\geq 0}P_{X_{k}^{\vect{n}}(Q)}(q)x^k
      &=F(x)\frac{1-x^2}{(1-qx)(1-q^{-1}x)}\\
      &=F(x)\left(\frac{1}{1-qx}+\frac{1}{1-q^{-1}x}-1\right)\\
      &=F(x)\left(1+\sum_{j\geq 1}(q^j+q^{-j})x^j\right)\\
      &=\sum_{k\geq 0}
      \left(a_k+\sum_{j=1}^k a_{k-j}(q^j+q^{-j})\right)x^k,
    \end{aligned}
  \end{equation}
  proving the first claim.

  To prove the second claim,
  we start with the case $|Q_0|=3$.
  Then $F(x)=1/(1-x)^3$, whose $i$-th coefficient is well known to equal $(i+1)(i+2)/2$.
  Passing to the case $|Q_0|=4$ amounts to taking consecutive sums of the $a_i$,
  and indeed $i(i+1)/2+(i+1)(i+2)/2=(i+1)^2$.
  Similarly, $i^2+(i+1)^2=2i(i+1)+1$, and the case $|Q_0|=5$ follows.
  For the case $|Q_0|=2$, we have to find a series of integers starting with $1$
  and with consecutive sums evaluating to $(i+1)(i+2)/2$,
  and indeed $\lfloor (i+1)^2/4\rfloor+\lfloor (i+2)^2/4\rfloor=(i+1)(i+2)/2$ is easily verified.
  This proves the second claim of the theorem.
\end{proof}

\begin{corollary}
  The generating function of Euler characteristics of the moduli spaces $X_k^{\vect{n}}(Q)$ is given by
  \begin{equation}
     \sum_{k\geq 0}\chi(X_k^{\vect{n}}(Q))x^k=\frac{(1+x)^{|Q_0|-2}}{(1-x)^4}.
  \end{equation}
\end{corollary}
 
\begin{proof}
  This follows from specializing $q$ to $1$ in (\ref{eqn:poincareseries}).
\end{proof}

In \cref{table:betti} we have given the even Betti numbers
from \cref{theorem:betti}
for~$k=1,\ldots,5$.

\begin{table}[ht]
  \centering
  \begin{tabular}{cc|cc}
    \toprule
    $Q$ & $k$ & even Betti numbers & Euler characteristic \\
    \midrule
    $\Aone$ & $1$ & $1,2,1$ & $4$ \\
    $\Aone$ & $2$ & $1,2,4,2,1$ & $10$ \\
    $\Aone$ & $3$ & $1,2,4,6,4,2,1$ & $20$ \\
    $\Aone$ & $4$ & $1,2,4,6,9,6,4,2,1$ & $35$ \\
    $\Aone$ & $5$ & $1,2,4,6,9,12,9,6,4,2,1$ & $56$ \\
    \midrule
    $\Atwo$ & $1$ & $1,3,1$ & $5$ \\
    $\Atwo$ & $2$ & $1,3,6,3,1$ & $14$ \\
    $\Atwo$ & $3$ & $1,3,6,10,6,3,1$ & $30$ \\
    $\Atwo$ & $4$ & $1,3,6,10,15,10,6,3,1$ & $55$ \\
    $\Atwo$ & $5$ & $1,3,6,10,15,21,15,10,6,3,1$ & $91$ \\
    \midrule
    $\Athree$ & $1$ & $1,4,1$ & $6$ \\
    $\Athree$ & $2$ & $1,4,9,4,1$ & $19$ \\
    $\Athree$ & $3$ & $1,4,9,16,9,4,1$ & $44$ \\
    $\Athree$ & $4$ & $1,4,9,16,25,16,9,4,1$ & $85$ \\
    $\Athree$ & $5$ & $1,4,9,16,25,36,25,16,9,4,1$ & $146$ \\
    \midrule
    $\Dfour$ & $1$ & $1,5,1$ & $7$ \\
    $\Dfour$ & $2$ & $1,5,13,5,1$ & $25$ \\
    $\Dfour$ & $3$ & $1,5,13,25,13,5,1$ & $63$ \\
    $\Dfour$ & $4$ & $1,5,13,25,41,25,13,5,1$ & $129$ \\
    $\Dfour$ & $5$ & $1,5,13,25,41,61,41,25,13,5,1$ & $231$ \\
    \bottomrule
  \end{tabular}
  \caption{The even Betti numbers and Euler characteristics of~$X_k^{\vect{n}}(Q)$ for~$k=1,\ldots,5$.}
  \label{table:betti}
\end{table}

\section{Cohomological Hall algebra action}
\label{section:coha}

The Euler form of $Q$ being symmetric when restricted to $\Lambda$ is also the setting in which
the so-called semistable cohomological Hall algebra,
a convolution structure on
\begin{equation}
  {\rm CoHa}^\partial(Q)\colonequals\bigoplus_{\vect{d}\in\Lambda}\HH^*_{{\rm G}_{\vect{d}}}(\Rep^{\partial\semistable}(Q,\vect{d}),\mathbb{Q}),
\end{equation}
is well-studied \cite{MR3779991}.
In the particular cases of \cref{proposition:5-series}, by \cite{MR5085658},
it is known to be isomorphic to an algebra~$\mathcal{P}_{|Q_0|-2}$ given by explicit generators and relations.
It admits a subalgebra
\begin{equation}
  {\rm CoHa}_\delta^\partial(Q)=\bigoplus_{k\geq 0}{\rm CoHa}^\partial(Q)_{k\delta}.
\end{equation}
This subalgebra naturally acts on the graded vector space
\begin{equation}
  \mathcal H_{\vect{n}}(Q)\colonequals\bigoplus_{k\geq 0}\HH^*(X_k^{\vect{n}}(Q),\mathbb{Q}),
\end{equation}
where we set~$X_0^{\vect{n}}(Q)=\mathrm{pt}$ to include the degree-zero term.
With this setup we can prove the following.

\begin{proposition}
  In the five cases of \cref{proposition:5-series},
  the direct sum~$\mathcal H_{\vect{n}}(Q)$ of
  the singular cohomologies of all~$X_k^{\vect{n}}(Q)$
  is a cyclic module over~${\rm CoHa}_\delta^\partial(Q)$, which is given by explicit generators and quadratic relations.
\end{proposition}

\begin{proof}
  By \cite[Theorem~3.2]{MR3779991}, the direct sum of cohomology~$\bigoplus_{\vect{d}\in\Lambda}\HH^*(X_{\vect{d}}^{\vect{n}}(Q),\mathbb{Q})$
  is a cyclic module over~${\rm CoHa}^\partial(Q)$ with generator~$1\in \mathbb{Q}\cong \HH^0(X_0^{\vect{n}}(Q),\mathbb{Q})$.
  In other words, the~$\Lambda$-graded map
  \begin{equation}
    {\rm CoHa}^\partial(Q)\rightarrow \bigoplus_{\vect{d}\in\Lambda}\HH^*(X_{\vect{d}}^{\vect{n}}(Q),\mathbb{Q}),
  \end{equation}
  given by multiplication by this class~$1$, is surjective.
  Consequently, the restriction~${\rm CoHa}_\delta^\partial(Q)\rightarrow \mathcal H_{\vect{n}}(Q)$ is also surjective,
  thus~$\mathcal H_{\vect{n}}(Q)$ is a cyclic module over~${\rm CoHa}_\delta^\partial(Q)$.
  The algebra~${\rm CoHa}^\partial(Q)$ is given (see \cite[\S6]{MR5085658}) by generators
\begin{itemize}
\item $e_{p,2k+1}$ for $p=1,\ldots,|Q_0|-2$, $k\geq 0$ in degree $\alpha_p$,
\item $f_{p,2k+1}$ for $p=1,\ldots,|Q_0|-2$, $k\geq 0$ in degree $\overline{\alpha}_p$,
\item $g_{2k}$ for $k\geq 0$ in degree $\delta$,
\item $h_{2k+2}$ for $k\geq 0$ in degree $\delta$.
\end{itemize}
Since
\begin{equation}
  [e_{p,2k+1},f_{p,2l+1}]=h_{2(k+l)+2}
\end{equation}
for all~$k,l\geq 0$ by \cite[\S6]{MR5085658},
the subalgebra~${\rm CoHa}_\delta^\partial(Q)$ is thus generated by the elements
\begin{itemize}
\item $e_{p,2k+1}f_{p,2l+1}$ for $p=1,\ldots,|Q_0|-2$ and $k,l\geq 0$,
\item $g_{2k}$ for $k\geq 0$,
\item $h_{2k+2}$ for $k\geq 0$,
\end{itemize}
all in degree~$\delta$.
These generators are subject to a list of quadratic relations also stated in \cite[\S6]{MR5085658}.
\end{proof}

\begin{remark}
  In fact, \cite[Theorem~3.2]{MR3779991} also describes the annihilator ideal of $1$ in ${\rm CoHa}^\partial(Q)$.
  Its intersection with ${\rm CoHa}_\delta^\partial(Q)$
  should be computable using the explicit description of ${\rm CoHa}^\partial(Q)$.
  Then, using the PBW basis for ${\rm CoHa}^\partial(Q)$ \cite[Proposition 6.1]{MR5085658},
  one should be able to compute the generating series of Betti numbers of the $X_k^{\vect{n}}(Q)$,
  therefore ``categorifying'' the proof of Theorem \ref{theorem:betti}.
  We do not do this here, since the present proof of the theorem, although being formal, is short and transparent.
\end{remark}

\section{Low-dimensional examples}
\label{section:low-dimensional}
We will now discuss the explicit geometry of these Fano varieties in more detail,
especially in low dimension.

Unlike for the Betti numbers,
where we have found a uniform description in \cref{theorem:betti},
we use the algorithms implemented in \cite{quivertools}
to compute some invariants,
namely in \cref{table:anticanonical-sections}
we give~$\mathrm{h}^0(X,\omega_X^\vee)$
for~$X=X_k^{\vect{n}}(Q)$ and~$k=2,3,4$,
and in \cref{table:chern-numbers}
we give the Chern numbers for~$k=2,3$
(where computationally feasible).
We do not list the values for~$k=1$,
because \cref{subsection:surfaces} provides explicit identifications.

Observe that~$\AoneA$ and~$\AoneB$ share all the numerical invariants
we list in \cref{table:anticanonical-sections,table:chern-numbers},
but by \cref{theorem:fano-series} they are not isomorphic,
as the index is~2, resp.~1.

\begin{table}[ht]
  \centering
  \begin{tabular}{ccccc}
    \toprule
    $k$ & 2 & 3 & 4 \\
    \midrule
    $\AoneA$  & 75 & 620 & 5140 \\
    $\AoneB$  & 75 & 620 & 5140 \\
    $\Atwo$   & 63 & 504 & 4090 \\
    $\Athree$ & 52 & 403 & 3206 \\
    $\Dfour$  & 42 & 316 & 2470 \\
    \bottomrule
  \end{tabular}
  \caption{Dimensions of~$\HH^0(X,\omega_X^\vee)$ for~$X=X_k^{\vect{n}}(Q)$.}
  \label{table:anticanonical-sections}
\end{table}

\begin{table}[ht]
  \centering
  \begin{subtable}{\textwidth}
    \centering
    \begin{tabular}{cccccc}
      \toprule
      $Q$ & $\mathrm{c}_1^4$ & $\mathrm{c}_2\mathrm{c}_1^2$ & $\mathrm{c}_2^2$ & $\mathrm{c}_3\mathrm{c}_1$ & $\mathrm{c}_4$ \\
      \midrule
      $\AoneA$  & 352 & 184 & 98 & 52 & 10 \\
      $\AoneB$  & 352 & 184 & 98 & 52 & 10 \\
      $\Atwo$   & 289 & 166 & 101 & 56 & 14 \\
      $\Athree$ & 232 & 148 & 107 & 58 & 19 \\
      $\Dfour$  & 181 & 130 & 116 & 58 & 25 \\
      \bottomrule
    \end{tabular}
    \caption{Fano fourfolds~$X_2^{\vect{n}}(Q)$.}
    \label{table:chern-numbers-fano-4-folds}
  \end{subtable}

  \medskip

  \begin{subtable}{\textwidth}
    \centering
    \begin{tabular}{cccccccccccc}
      \toprule
      $Q$ & $\mathrm{c}_1^6$ & $\mathrm{c}_2\mathrm{c}_1^4$ & $\mathrm{c}_2^2\mathrm{c}_1^2$ & $\mathrm{c}_2^3$ & $\mathrm{c}_3\mathrm{c}_1^3$ & $\mathrm{c}_3\mathrm{c}_2\mathrm{c}_1$ & $\mathrm{c}_3^2$ & $\mathrm{c}_4\mathrm{c}_1^2$ & $\mathrm{c}_4\mathrm{c}_2$ & $\mathrm{c}_5\mathrm{c}_1$ & $\mathrm{c}_6$ \\
      \midrule
      $\AoneA$  & 38272 & 19968 & 10408 & 5436 & 6480 & 3372 & 1088 & 1368 & 724 & 192 & 20 \\
      $\AoneB$  & 38272 & 19968 & 10408 & 5436 & 6480 & 3372 & 1088 & 1368 & 724 & 192 & 20 \\
      $\Atwo$   & 30287 & 16539 & 9062 & 5030 & 5805 & 3190 & 1122 & 1404 & 797 & 234 & 30 \\
      \bottomrule
    \end{tabular}
    \caption{Fano sixfolds~$X_3^{\vect{n}}(Q)$. Values for~$\Athree$ and~$\Dfour$ are missing because of computational limits.}
    \label{table:chern-numbers-fano-6-folds}
  \end{subtable}

  \caption{Chern numbers of~$X_k^{\vect{n}}(Q)$ for~$k=2,3$.}
  \label{table:chern-numbers}
\end{table}

In dimensions~2 and~4,
so when~$k=1,2$,
we will describe their geometry in a bit more detail.

\subsection{del Pezzo surfaces}
\label{subsection:surfaces}
We can identify the del Pezzo surfaces of Picard rank between~$2$ and~$5$
with the $k=1$ cases of the above five series,
giving:
\begin{center}
  \begin{tabular}{cc}
    \toprule
    $Q$ & $X_1^{\vect{n}}(Q)$ \\
    \midrule
    $\AoneA$  & $\mathbb{P}^1\times\mathbb{P}^1$ \\
    $\AoneB$  & $\Bl_1\mathbb{P}^2$ \\
    $\Atwo$   & $\Bl_2\mathbb{P}^2$ \\
    $\Athree$ & $\Bl_3\mathbb{P}^2$ \\
    $\Dfour$  & $\Bl_4\mathbb{P}^2$ \\
    \bottomrule
  \end{tabular}
\end{center}
This identification is standard,
and in our setting follows from inspecting the index and the Picard rank,
as computed in \cref{theorem:fano-series}.

We can make \cref{proposition:fibers-forgetful}
explicit in this easy case,
and obtain the conic bundle descriptions of these del Pezzo surfaces.
Namely,
the Luna stratification of the base~$\mathbb{P}^1$
consists of the generic statum attached to the Luna type~$(\delta^1)$,
where the fiber is a smooth conic
(we have~$s=1$ with~$k_1=1$, and~$\ell_i=0$ for~$i=1,\ldots,\#Q_0-2$),
and~$\#Q_0-2$ strata (which are points in~$\mathbb{P}^1)$
where the fiber is the nodal conic
(we have~$s=0$, and~$\ell_i=1$ for exactly one~$i\in\{1,\ldots,\#Q_0-2\}$).

\subsection{Fano fourfolds}
\label{subsection:fourfolds}

Next, we take~$k=2$, and obtain five Fano fourfolds.
From the discussion in \cref{proposition:fibers-forgetful}
we see that the morphism~$p_2\colon X\to\mathbb{P}^2$
looks generically like a quadric surface bundle,
but with singular fibres which are not merely singular quadric surfaces.
Recently, Kresch--Tschinkel studied such fibrations,
and we recall their \cite[Definition~1]{MR4159824}.

\begin{definition}
  \label{definition:involution-surface-bundle}
  Let~$S$ be a smooth variety.
  An \emph{involution surface bundle} is
  a flat projective morphism~$p\colon X\to S$ for which
  \begin{itemize}
    \item the locus~$U\subseteq S$ where~$p$ is a smooth morphism is Zariski-dense,
    \item the fibre~$p^{-1}(s)$ for~$s\in U$ is an involution surface,
      i.e., a minimal del Pezzo surface of degree~8.
  \end{itemize}
  We say it has \emph{mild degeneration}
  if every singular fibre is geometrically isomorphic to
  one of the following reduced schemes:
  \begin{description}
    \item[type I]
      a nodal quadric surface
    \item[type II]
      the product of two reduced singular conics
    \item[type III]
      the union of two copies of the second Hirzebruch surface~$\mathbb{F}_2$,
      with each~$(-2)$-curve glued to a fibre of the other surface
    \item[type IV]
      the product of a smooth conic with a reduced singular conic.
  \end{description}
\end{definition}

We will extend \cite[Theorem~24]{2607.12895},
which is only concerned with~$\Dfour$.
We can reuse parts of the computation.
First,
recall that the possible Luna types are determined in \cref{lemma:luna-types},
and that \cref{lemma:local-quivers} gives the local quivers.
For all these local quivers
we list the fibers from \cref{proposition:fibers-forgetful}
and a reference for their description,
in \cref{table:luna-types-fibers},
where~$\alpha,\beta\in P$ are distinct and not complementary when both occur.
\begin{table}[ht]
  \centering
  \begin{tabular}{lll}
    \toprule
    Luna type & fibre & reference \\
    \midrule
    $(\delta,\delta)$ & $\mathbb{P}^1\times\mathbb{P}^1$ & \cite[Lemma~18]{2607.12895} \\
    $(\delta^2)$ & type~I & \cite[Lemma~21]{2607.12895} \\
    $(\delta,\alpha,\overline{\alpha})$ & type~IV & \cite[Lemma~19]{2607.12895} \\
    $(\alpha,\overline{\alpha},\beta,\overline{\beta})$ & type~II & \cite[Lemma~20]{2607.12895} \\
    $(\alpha^2,\overline{\alpha}^2)$ & type~III & \cite[Lemma~22]{2607.12895} \\
    \bottomrule
  \end{tabular}
  \caption{Luna types and fibres of~$p_2$ for~$k=2$.}
  \label{table:luna-types-fibers}
\end{table}

This allows us to prove the following generalization of \cite[Theorem~24]{2607.12895}.
\begin{proposition}
  \label{proposition:involution-surface-bundles}
  Let~$Q$ be one of the extended Dynkin quivers from \cref{proposition:5-series}.
  The morphism
  \begin{equation}
    p_2\colon X_2^{\vect{n}}(Q)\to\modulispace[\partial\semistable]{Q,2\delta}\cong\mathbb{P}^2
  \end{equation}
  is an involution surface bundle which has mild degeneration.
  The discriminants are the following subsets of the configuration
  in \cref{figure:discriminant-locus}:
  \begin{center}
    \begin{tabular}{ll}
      $Q$ & discriminant \\
      \midrule
      $\AoneA$  & the conic \\
      $\AoneB$  & the conic \\
      $\Atwo$   & the conic and one line \\
      $\Athree$ & the conic and two lines \\
      $\Dfour$  & the conic and the triangle of lines \\
    \end{tabular}
  \end{center}
\end{proposition}

\begin{proof}
  By \cref{lemma:luna-types}, the Luna types of representations of dimension~$2\delta$
  are precisely those in the first column of \cref{table:luna-types-fibers}.
  The local-quiver description in \cref{proposition:fibers-forgetful}
  and the computations in \cite[Lemmas~18--22]{2607.12895}
  gives the corresponding fibre descriptions
  in the second column of \cref{table:luna-types-fibers}.

  Under the identification
  $\modulispace[\partial\semistable]{Q,2\delta}\cong\Sym^2\mathbb{P}^1\cong\mathbb{P}^2$,
  the indicated Luna strata give the conic and lines in the discriminant table.
  The source is smooth and projective by \cref{theorem:fano-series},
  and the preceding equidimensionality result gives flatness.
  The general fibre is~$\mathbb{P}^1\times\mathbb{P}^1$,
  so the result follows.
\end{proof}

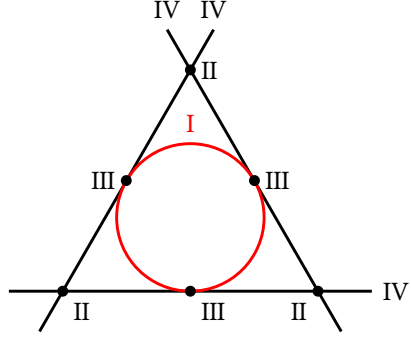
\begin{figure}[t]
  \centering
  \begin{tikzpicture}[scale=0.75, font=\small, line width=1.1pt]
    \draw (-3.2,-1.3) -- (3.2,-1.3)    node[right] {IV};
    \draw (2.66,-2.01) -- (-0.41,3.31) node[above] {IV};
    \draw (-2.66,-2.01) -- (0.41,3.31) node[above] {IV};
    \draw[red] (0,0) circle (1.3);
    \node[red] at (0,1.65) {I};
    \filldraw (0,2.6)      circle (2pt) node[right]       {II};
    \filldraw (-2.25,-1.3) circle (2pt) node[below right] {II};
    \filldraw (2.25,-1.3)  circle (2pt) node[below left]  {II};
    \filldraw (0,-1.3)     circle (2pt) node[below right] {III};
    \filldraw (1.13,0.65)  circle (2pt) node[right]       {III};
    \filldraw (-1.13,0.65) circle (2pt) node[left]        {III};
  \end{tikzpicture}
  \caption{Configuration of three lines and a smooth conic tangent to each of them,
    labeled by the possible fibers in the involution surface bundle structures
    described in \cref{proposition:involution-surface-bundles}.
    Taken from \cite[Figure~1]{2607.12895}.
  }
  \label{figure:discriminant-locus}
\end{figure}

\printbibliography

\emph{Pieter Belmans}, \url{p.belmans@uu.nl} \\
Mathematical Institute, Utrecht University, Budapestlaan 6, 3584CD Utrecht, Netherlands

\emph{Markus Reineke}, \url{markus.reineke@rub.de} \\
Faculty of Mathematics, Ruhr-Universit\"at Bochum, Universit\"atsstra\ss e 150, 44780 Bochum, Germany

\end{document}